\documentclass[a4paper,11pt]{amsart}
\usepackage[utf8]{inputenc}
\usepackage{amsmath,amssymb,amsfonts}
\usepackage{epsfig}
\usepackage{mathrsfs}
\usepackage{graphicx,color}
\usepackage{eucal}

\def\ind{\mathrel{\hbox{\rlap{%
\hbox to 7.5pt{\hrulefill}}\raise6.6pt\hbox{\eka\char'167}}}}
\newtheorem{theorem}{Theorem}
\newtheorem{proposition}{Proposition}
\newtheorem{corollary}{Corollary}

\newtheorem{lemma}{Lemma}
\newtheorem{example}{Example}

\newcommand{\Ex}[1]{\mathbb{E}\! \left(#1\right)}
\newcommand{\Cov}[1]{\textup{Cov}\left[#1\right]}
\newcommand{\ip}[1]{\left<#1\right>}
\newcommand{\dom}{\mathcal{D}}          
\newcommand{\norm}[1]{\left| \! \left| #1 \right| \! \right|}
\newcommand{\bbm}[1]{\left[\begin{matrix} #1 \end{matrix}\right]}
\newcommand{\sbm}[1]{\left[\begin{smallmatrix} #1
             \end{smallmatrix}\right]}

\newcommand{\sumk}{\sum_{k=1}^{\infty}}

\newcommand{\tr}{\textup{tr}}
\newcommand{\sX}{\mathcal{X}}
\newcommand{\sY}{\mathcal{Y}}
\newcommand{\sH}{\mathcal{H}}
\newcommand{\sU}{\mathcal{U}}
\newcommand{\Ch}{C_{\square}}

\begin{document}



\title[]{Convergence of discrete time Kalman filter estimate to
  continuous time estimate*\footnote{*T\lowercase{his is the accepted author's version of the manuscript accepted for publication in} I\lowercase{nternational} J\lowercase{ournal of} c\lowercase{ontrol.}}}

\author[Atte Aalto]{Atte Aalto$^{\lowercase{\textup{a,b}}}$ \\ \\ $^{\lowercase{\textup{a}}}$D\lowercase{epartment of} M\lowercase{athematics and} S\lowercase{ystems} A\lowercase{nalysis}, A\lowercase{alto} U\lowercase{niversity}, E\lowercase{spoo}, F\lowercase{inland} \\ $^{\lowercase{\textup{b}}}$I\lowercase{nria}, U\lowercase{niversit\'e} P\lowercase{aris}--S\lowercase{aclay}, P\lowercase{alaiseau}, F\lowercase{rance}; M$\Xi$DISIM \lowercase{team} }
  

\maketitle

\begin{abstract}
  This article is concerned with the convergence of the
  state estimate obtained from the discrete time Kalman filter to the
  continuous time estimate as the temporal discretization is
  refined. The convergence follows from Martingale convergence theorem
  as demonstrated below but, surprisingly, no results exist on the
  rate of convergence. We derive convergence rate estimates for the discrete time Kalman filter estimate for finite and infinite dimensional systems. The proofs
  are based on applying the discrete time Kalman filter on a dense
  numerable subset of a certain time interval
  $[0,T]$.
  
  \medskip
\noindent Keywords: Kalman filter, infinite dimensional systems,
  temporal discretization, sampled data
\end{abstract}

\section{Introduction}

It is well known that Kalman filter (or Kalman--Bucy filter) gives the
optimal solution to the state estimation problem for discrete (or continuous) time
linear systems with Gaussian initial state, and Gaussian input and
output noise processes. These filters have proven to be
robust and they have been widely used in practical applications since
their introduction in the 1960s. The implementation of the discrete
time filter is straightforward since it is readily formulated in an
algorithmic manner. Thus, it may often be tempting to use the discrete
time filter on the temporally discretized continuous time system. The
purpose of this article is to study the convergence of a state
estimate from discrete time Kalman filter to the continuous time state
estimate as the temporal discretization is refined. In particular, we
show convergence speed estimates for the quadratic error between the
discrete time and continuous time estimate first for finite
dimensional systems without input noise, then finite dimensional systems with input noise, and finally, for infinite dimensional systems with a bounded observation operator.

The class of systems studied here is described
by mappings $(A,B,C)$ where $A:\sX \to \sX$, $B:\sU \to \sX$, and $C:\sX \to \sY$,
and the
corresponding dynamics equations
\begin{equation} \label{eq:system}
\begin{cases}
dz(t)=Az(t) \, dt + Bdu(t), \qquad t \in \mathbb{R}^+, \\
dy(t)=Cz(t) \, dt + dw(t), \\
z(0)=x.
\end{cases}
\end{equation}
Here $\sX$ is called the \emph{state space}, $\sU = \mathbb{R}^q$ is the \emph{input space}, and $\sY=\mathbb{R}^r$ is the
\emph{output space}. The mapping $A$ is the generator of a contractive
$C_0$-semigroup $e^{At}$ on $\sX$ with domain $\dom(A)$, $B:\mathbb{R}^q \to \sX$ is the \emph{input operator},  and $C:\sX \to
\mathbb{R}^r$ is called the \emph{observation operator}.
The observation operator can be bounded or not but it always maps to a
finite dimensional space in this article. The process $y$ is called the
\emph{output process}. The \emph{input} and \emph{output noise processes} $u$ and $w$ are assumed
to be $q$- and $r$-dimensional Brownian motions with incremental covariance
matrices $Q>0$ and $R>0$, respectively, and the \emph{initial state} $x$ is assumed to be an
$\sX$-valued Gaussian random variable, $x \sim N(m,P_0)$. The noise processes $u$ and $w$ and the initial state $x$ are assumed to be mutually independent. Note that the system \eqref{eq:system} is written as a stochastic differential equation. For background of stochastic equations and the formulation of the Kalman--Bucy filter in this framework, we refer to \cite{Oks} (in particular, Section~6.3 therein) and \cite{Davis}.

The discrete and continuous time state estimates are defined by
\begin{equation} \label{eq:estimates} \hat{z}_{T,n}:=\Ex{z(T) \,
    \Big|\left\{y\left(\tfrac{iT}{n}\right)\right\}_{i=1}^n} \ \quad
  \textrm{and} \ \quad \hat{z}(T):=\Ex{z(T) \, \big|\big\{y(s),s \le T
    \big\}},
\end{equation}
respectively. That is, we are estimating the final state of the
system \eqref{eq:system}. 
These estimates
are given by the discrete and continuous time Kalman filter,
respectively.
 The purpose of this article is to study the
convergence $\hat{z}_{T,n} \to \hat{z}(T)$ as $n \to \infty$.

In Section~\ref{sec:stochastics}, we cover the necessary background
concerning stochastics and the Kalman filter. 
The proofs of the main results are based on using the discrete time Kalman filter on a sequence that forms a dense subset of the interval $[0,T]$.
In particular, in Section~\ref{sub:stochastics}, it is shown that  this procedure in fact converges to $\hat{z}(T)$ strongly in $\sX$ almost
surely. Gaussian random variables and the Kalman filter are introduced
in Section~\ref{sub:Kalman}. Section~\ref{sec:noiseless} contains the main result in the simplest case, namely namely an estimate of the convergence speed of $\Ex{\norm{\hat{z}_{T,n}
    - \hat{z}(T)}_{\sX}^2}$ when $n$ is increased for finite dimensional system without input noise. The proofs of the other results follow the same outline, and so this simplest case is shown in full detail in order to convey the ideas as clearly as possible.
In the beginning of the section it is shown how to take into account an intermediate measurement in Kalman filtering --- an important tool in the proofs. The result for systems with input noise is shown in Section~\ref{sec:noise} and for infinite dimensional systems with bounded observation operator the result is generalized in Section~\ref{sec:infdim}.


The Kalman filter performance has been widely studied in
literature. Even though it was originally derived for state estimation
for finite dimensional linear systems with Gaussian input and output
noise processes it has proven to be very robust and thus applicable to
a variety of other scenarios. Variants for non-linear systems have
been developed, such as the extended Kalman filter
and the unscented Kalman filter, see the book \cite{simon_book}. Kalman filter sensitivity to modelling errors has been studied by for example \cite{Sun} and
\cite[Chapter~7]{Gelb}. See also the recent work \cite{Wave_error} for
a study on the effect of modelling errors in an infinite dimensional
example case, namely the one dimensional wave equation. The effect of
state space discretization to Kalman filtering has been studied in,
\emph{e.g.}, \cite{Bensoussan}, \cite{Galerkin_filter}, and in \cite{Aalto_state}.

However, the error that stems from using the discrete time filter on
the temporally discretized continuous time system has not received
much attention.  Two recent articles, \cite{Axelsson} and
\cite{Axelsson2}, have studied different numerical methods for
approximating the matrix exponential $e^{A\Delta t}$ and the effect of
this approximation on the solution of the corresponding Lyapunov
equations and Kalman filtering. A convergence result of the discrete
time Kalman filter estimate in finite dimensional setting is shown by
 \cite{Salgado} without convergence rate
estimate. They use similar techniques that can also be used to
(formally) obtain the Kalman-Bucy filter as a limit of the discrete
time Kalman filter, as is done for example in
\cite[Section~8.2]{simon_book} and \cite[Section~4.3]{Gelb}.

\subsubsection*{Notation and standing assumptions}
\begin{itemize}
\item[$\circ$] The space of bounded operators from a Hilbert space
  $\sH_1$ to another Hilbert space $\sH_2$ is denoted by
  $\mathcal{L}(\sH_1,\sH_2)$, and
  $\mathcal{L}(\sH_1)=\mathcal{L}(\sH_1,\sH_1)$.
\item[$\circ$] We assume that the state space $\sX$ is a separable
  Hilbert space. Denote by $\{e_k\}_{k=1}^{p/\infty} \subset \sX$ an
  orthonormal basis for the $p/\infty$-dimensional state space.
\item[$\circ$] $A$ is the generator of a $C_0$-semigroup
  on $\sX$. The semigroup is denoted by $e^{At}$ even though $A$ is
  not bounded in general. We assume $\norm{e^{At}}_{\mathcal{L}(\sX)} \le \mu$ for $t \in [0,T]$.
\item[$\circ$] The space $\dom(A)$ is equipped with the graph norm
  $\norm{x}_{\dom(A)}^2=\norm{x}_{\sX}^2+\norm{Ax}_{\sX}^2$ which makes
  $\dom(A)$ a Hilbert space since $A$ is closed.
\item[$\circ$] We assume that the observation operator is bounded, $C \in \mathcal{L}(\sX,\sY)$, and that the input operator is smooth, that is, $B \in \mathcal{L}(\sU,\dom(A))$. The input and output spaces are always  finite dimensional, $\sU=\mathbb{R}^q$ and $\sY=\mathbb{R}^r$.
\item[$\circ$] $\Omega$ is a probability space and $L^2(\Omega;\sX)$
  is the space of $\sX$-valued random variables $\xi$ satisfying
  $\Ex{\norm{\xi}_{\sX}^2}<\infty$.
\item[$\circ$] The sigma algebra generated by a random variable $\xi$ is
  denoted by $\sigma\{ \xi \}$.
\item[$\circ$] To improve readability, we use index $n$ only when
  referring to the discretization level in the state estimate $\hat
  z_{T,n}$ defined in \eqref{eq:estimates}, index $k$ only to denote
  different dimensions of the state space, and index $j$ only when
  referring to the martingale $\tilde z_j$ defined below in
  Section~\ref{sub:stochastics}.

\end{itemize}

\section{Background and preliminary results} \label{sec:stochastics}

As mentioned above, the proofs of this article are based on applying
the discrete time Kalman filter on a dense, numerable subset on the
interval $[0,T]$ --- starting from the discrete time state estimate
$\hat z_{T,n}$ --- and computing an upper bound for the change in the
estimate. In section \ref{sub:stochastics}, we establish that the
limit thus obtained is indeed $\hat z(T)$. Gaussian random variables
and the Kalman filter are discussed in Section~\ref{sub:Kalman}. 


\subsection{Stochastics} \label{sub:stochastics}

In the cases where the state space $\sX$ is infinite dimensional it is
always assumed either that $x \in \dom(A)$ almost surely or that $C
\in \mathcal{L}(\sX,\sY)$. This guarantees that the stochastic process
$y$ given by \eqref{eq:system} has almost surely continuous sample
paths. Let $\left\{t_i \right\}_{i=1}^{\infty}$ be a dense subset of
the interval $[0,T]$ and denote ${\textup T}_j:=\left\{t_i
\right\}_{i=1}^j$.  Now let $\xi$ be an integrable $\sX$-valued random
variable and $y$ a stochastic process with almost surely continuous
sample paths. Then $[\xi]_k:=\ip{\xi,e_k}_{\sX}$ is an integrable
$\mathbb{R}$-valued random variable for each $k$. Define the
martingales $[\tilde \xi_j]_k:=\Ex{\ip{\xi,e_k}_{\sX}|\mathcal{F}_j}$
where $\mathcal{F}_j$ is the sigma algebra generated by $\{y(t),t \in
{\textup T}_j\}$, that is, $\mathcal{F}_j=\sigma\left\{ y(t),t \in
  {\textup T}_j \right\}$. It holds that $\Ex{|[\tilde \xi_j]_k|} \le
\Ex{| \!  \ip{\xi,e_k}_{\sX} \! |}$ for all $j$ and thus by Doob's
Martingale convergence theorem (see \cite[Appendix~C]{Oks}, in
particular, Theorem~C.6 and Corollary~C.9), $[\tilde \xi_j]_k \to
[\tilde \xi_{\infty}]_k$ almost surely. As $y$ has continuous sample
paths, it holds that $[\tilde \xi_{\infty}]_k=
\Ex{\ip{\xi,e_k}_{\sX}|\{y(s),s \le T \}}$ almost surely. Using this
componentwise implies that $\tilde \xi_j
:=\Ex{\xi|\mathcal{F}_j}=\sum_{k=1}^{\infty} [\tilde \xi_j]_ke_k$
converges strongly (in $\sX$) almost surely to $\tilde
\xi_{\infty}=\sum_{k=1}^{\infty} [\tilde \xi_{\infty}]_ke_k$.

In general, the martingale convergence theorem is true for Banach spaces that have the Radon--Nikodym property. All reflexive Banach spaces (and therefore also Hilbert spaces) have the Radon--Nikodym property. The above deduction follows essentially the proof of this fact in the special case of $\sX$ being a Hilbert space, see \cite[Corollary~2.11]{Pisier}.


In the proofs, we will need the following telescope identity for martingales.
\begin{lemma} \label{lem:tele}
Let $\xi_j$ be a square integrable $\sX$-valued martingale. Then for 
$L,N \in \mathbb{N}$ with
  $L \ge N$:
\begin{equation} \nonumber
\Ex{\norm{\xi_L-\xi_N}_{\sX}^2}=\sum_{j=N}^{L-1} \Ex{\norm{\xi_{j+1}-\xi_j}_{\sX}^2}.
\end{equation}
\end{lemma}
\begin{proof}
The result follows directly from the fact that martingale increments are orthogonal. Let us show this. Let $k \ge j$ and denote $\mathcal{F}_i=\sigma \{ \xi_1,...,\xi_i \}$. Then (recalling that $\xi_k-\Ex{\xi_k|\mathcal{F}_j} \perp \xi_i$ for $i \le j$ and the martingale property $\Ex{\xi_k|\mathcal{F}_j}=\xi_j$),
\[
\Ex{\ip{\xi_k,\xi_j}_{\sX}}=\Ex{\ip{\Ex{\xi_k|\mathcal{F}_j}+(\xi_k-\Ex{\xi_k|\mathcal{F}_j}),\xi_j}_{\sX}}=\Ex{\ip{\xi_j,\xi_j}_{\sX}}.
\]
Using this, we have (let now $k>j$)
\begin{align*}
\Ex{\ip{\xi_{k+1}-\xi_k,\xi_{j+1}-\xi_j}}=&\Ex{\ip{\xi_{k+1},\xi_{j+1}}}-\Ex{\ip{\xi_{k+1},\xi_j}} \\ & -\Ex{\ip{\xi_k,\xi_{j+1}}}+\Ex{\ip{\xi_k,\xi_j}}=0.
\end{align*}
\end{proof}

Below we sometimes need the assumption that $x \in \dom(A)$ almost
surely. With Gaussian random variables this means that $x$ is actually
a $\dom(A)$-valued random variable.
\begin{proposition} \label{prop:X_1} Let $\xi$ be an $\sX$-valued
  Gaussian random variable s.t. $\xi \in \sX_1$ almost surely where
  $\sX_1 \subset \sX$ is another Hilbert space with continuous and
  dense embedding. Then $\xi$ is an $\sX_1$-valued Gaussian random
  variable.
\end{proposition}
\begin{proof}
  Pick $h \in \sX_1$. We intend to show that $\ip{\xi,h}_{\sX_1}$ is a
  real-valued Gaussian random variable. For $h \in \sX_1$ there exists
  $h' \in \sX_1'$, the dual space of $\sX_1$,
  s.t. $\ip{\xi,h}_{\sX_1}=\ip{\xi,h'}_{(\sX_1,\sX_1')}$ and further,
  there exists a sequence $\{h_i\}_{i=1}^{\infty} \subset \sX$ such
  that $\ip{\xi,h'}_{(\sX_1,\sX_1')}=\lim_{i \to
    \infty}\ip{\xi,h_i}_{\sX}$. Now $\ip{\xi,h_i}_{\sX}$ is a pointwise
  converging sequence of Gaussian random variables and so the limit is
  also Gaussian.
\end{proof}
\noindent Fernique's theorem \cite[Theorem~2.6]{DaPrato} can be
applied to note that
if $\xi$ is an $\sX_1$-valued Gaussian random variable then $\xi \in
L^p(\Omega;\sX_1)$ for any $p > 0$. In particular,
$\Ex{\norm{\xi}_{\sX_1}^2}<\infty$ and if $A \in \mathcal{L}(\sX_1,\sX)$ then
$A\xi$ is an $\sX$-valued Gaussian random variable.

\subsection{Kalman filter} \label{sub:Kalman}

The discrete time Kalman filter was originally presented in
\cite{Kalman}. The continuous time filter is known as the Kalman--Bucy
filter, and it was presented in \cite{Kalman_Bucy}. We also refer to
the book \cite{Gelb} for a thorough introduction to both discrete and
continuous time Kalman filters as well as the usual techniques needed
in different scenarios.  Of course, the original presentations are in
finite dimensional setting. The infinite dimensional generalization of
the discrete time Kalman filter is rather straightforward, and it can
be found for example in \cite{Horowitz_phd}. The infinite dimensional
Kalman--Bucy filter is considered in \cite{Bensoussan} and \cite[Chapter~6]{CP}. 
However, we do not need to be concerned with the continuous time equations.
Our approach is based on
using the discrete time Kalman filter on a numerable set $\{ t_j \}_{j=1}^{\infty}$
that is dense on an interval $[0,T]$, and bounding the $L^2(\Omega;\sX)$-norm of the estimate increment when adding a new time point $t_j$. In this section we thus review the discrete time
Kalman filter equations.

The Kalman filter is based on the fact that with linear systems with
Gaussian initial state and input and output noise processes, the state
vector remains a Gaussian stochastic process. Also, the conditional
expectation of the state with respect to the measurements is a
Gaussian process. The statistical properties of the Gaussian
$\sX$-valued random variable $\xi$ are completely characterized by the
mean $m=\Ex{\xi} \in \sX$ and the covariance operator $P=\Cov{\xi,\xi} \in
\mathcal{L}(\sX)$, defined for $h \in \sX$ by $\Cov{\xi,\xi}h:=\Ex{(\xi-m)
  \ip{\xi-m,h}_{\sX}}$. Thus it is meaningful to write $\xi \sim N(m,P)$
meaning that $\xi$ is a Gaussian random variable with mean $m$ and
covariance $P$. The covariance operator is symmetric and nonnegative
and, in addition, it is a trace class operator with
$\tr(P)=\Ex{\norm{\xi-m}_{\sX}^2}$, see \cite[Lemma~2.14 \&
Proposition~2.15]{DaPrato}. In fact, by Fernique's theorem, Gaussian
random variables are $p$-integrable for every $p > 0$.

For square integrable random variables, the conditional expectation
with respect to a random variable $\xi$ is a projection onto the
subspace generated by $\xi$. With jointly Gaussian random variables $\xi_1
\in \sX$ and finite dimensional $\xi_2$, this projection has an easy
representation. That is, if $\xi=\bbm{\xi_1 \\ \xi_2} \sim N\left( \bbm{m_1
    \\ m_2}, \bbm{P_{11} & P_{12} \\ P_{12}^* & P_{22}} \right)$ then
\begin{equation} \nonumber
\Ex{\xi_1|\xi_2}=m_1+P_{12}P_{22}^+(\xi_2-m_2)
\end{equation}
where $P_{22}^+$ denotes the (Moore-Penrose) pseudoinverse of
$P_{22}$. The error covariance is
\begin{equation} \nonumber
\Cov{\xi_1-\Ex{\xi_1|\xi_2},\xi_1-\Ex{\xi_1|\xi_2}}=P_{11}-P_{12}P_{22}^+P_{12}^*.
\end{equation}
Now applying the above equations to a Gaussian random variable
$[\xi_1,\xi_2,\xi_3]$ where $\xi_2$ and $\xi_3$ are finite dimensional, and the
2-by-2 blockwise matrix inversion formula to $\Cov{\sbm{\xi_2 \\
    \xi_3},\sbm{\xi_2 \\ \xi_3}}$ leads directly to
\begin{align} \label{eq:KF_m}
\Ex{\xi_1|[\xi_2,\xi_3]}=&\Ex{\xi_1|\xi_2}+\Cov{\xi_1-\Ex{\xi_1|\xi_2},\xi_3-\Ex{\xi_3|\xi_2}}\times \\
\nonumber &\times \Cov{\xi_3-\Ex{\xi_3|\xi_2},\xi_3-\Ex{\xi_3|\xi_2}}^+(\xi_3-\Ex{\xi_3|\xi_2})
\end{align}
and
\begin{align} \nonumber
  &\Cov{\xi_1-\Ex{\xi_1|[\xi_2,\xi_3]},\xi_1-\Ex{\xi_1|[\xi_2,\xi_3]}} \\
   \label{eq:KF_P} &  =\Cov{\xi_1-\Ex{\xi_1|\xi_2},\xi_1-\Ex{\xi_1|\xi_2}}  
-\Cov{\xi_1-\Ex{\xi_1|\xi_2},\xi_3-\Ex{\xi_3|\xi_2}} \\
  \nonumber & \hspace{17mm}
  \times\Cov{\xi_3-\Ex{\xi_3|\xi_2},\xi_3-\Ex{\xi_3|\xi_2}}^+\Cov{\xi_3-\Ex{\xi_3|\xi_2},\xi_1-\Ex{\xi_1|\xi_2}}.  
\end{align}
These equations make it possible to update the state estimate (here
$\Ex{\xi_1|\xi_2}$) recursively when a new measurement (here $\xi_3$) is
obtained from the system.

From \eqref{eq:KF_m} we get the covariance for the increment
$\Ex{\xi_1|[\xi_2,\xi_3]}-\Ex{\xi_1|\xi_2}$,
\begin{align*}
  &\Cov{\Ex{\xi_1|[\xi_2,\xi_3]}-\Ex{\xi_1|\xi_2},\Ex{\xi_1|[\xi_2,\xi_3]}-\Ex{\xi_1|\xi_2}} \\
  &=
  \Cov{\xi_1-\Ex{\xi_1|\xi_2},\xi_3-\Ex{\xi_3|\xi_2}}\Cov{\xi_3-\Ex{\xi_3|\xi_2},\xi_3-\Ex{\xi_3|\xi_2}}^+
  \\ & \hspace{60mm} \times \Cov{\xi_3-\Ex{\xi_3|\xi_2},\xi_1-\Ex{\xi_1|\xi_2}},
\end{align*}
and further, the $L^2(\Omega;\sX)$-norm of the increment is given by
\begin{align} \nonumber
  &\Ex{\norm{\Ex{\xi_1|[\xi_2,\xi_3]}-\Ex{\xi_1|\xi_2}}_{\sX}^2}\\  \label{eq:Lincr}  &=\tr\Big(
  \Cov{\xi_1-\Ex{\xi_1|\xi_2},\xi_3-\Ex{\xi_3|\xi_2}}\Cov{\xi_3-\Ex{\xi_3|\xi_2},\xi_3-\Ex{\xi_3|\xi_2}}^+
  \\ \nonumber & \hspace{60mm} \times \Cov{\xi_3-\Ex{\xi_3|\xi_2},\xi_1-\Ex{\xi_1|\xi_2}}
  \Big).
\end{align}
This fact will be used multiple times in the proofs below.

The familiar discrete time Kalman filter equations follow directly
from \eqref{eq:KF_m} and \eqref{eq:KF_P} if $\xi_1$ is chosen to be the
current state $x_i$ that is to be estimated, $\xi_2$ consists of the old
outputs $[y_1,\dots,y_{i-1}]$, and $\xi_3$ is the new output $y_i$.

\section{The case without input noise} \label{sec:noiseless}

For simplicity of presentation, let us first go through the case without input noise. In this case the solution to \eqref{eq:system} is simply $z(t)=e^{At}x$.


The convergence rate estimates are based on computing how much $\hat
z_{T,n}$ can change at most (measured with the $L^2(\Omega ;
\sX)$-norm) when more and more output values $y(t)$ are taken into
account from the intervals $t \in ((i-1)T/n,iT/n)$ for $i=1,\dots,n$. In
this section, it is first shown how an intermediate measurement is taken
into account.  Consider the output of the system \eqref{eq:system},
$dy(t)=Ce^{At}x \, dt + dw(t)$, which is a shortened notation for
\begin{equation} \label{eq:output}
  y(t)=C \int_0^t e^{As}x \, ds + w(t)
\end{equation}
where $A$ and $C$ are operators from $\sX$ to
$\sX$ and $\sY=\mathbb{R}^r$, respectively, and $w$ is an $r$-dimensional Brownian
motion with incremental covariance matrix $R$. 

Assume we have a state estimate $\tilde x_j:=\Ex{x|\{y(t_1),
  y(t_2),\dots,y(t_j)\}}$ for the initial state $x$, and the
corresponding error covariance $P_j:=\Cov{x-\tilde x_j,x-\tilde
  x_j}$. Now the next measurement to be taken into account in state
estimation is $y(t_{j+1})$. Say $t_{j+1} \in (t_a,t_b)$ for some $a,b
\in \{1,\dots,j\}$ and that this interval does not contain any earlier
included measurements, that is $t_i \notin (t_a,t_b)$ for
$i=1,\dots,j$. The new state estimate $\tilde x_{j+1}$ and the
corresponding error covariance $P_{j+1}:=\Cov{x-\tilde
  x_{j+1},x-\tilde x_{j+1}}$ are given by \eqref{eq:KF_m} and
\eqref{eq:KF_P}, respectively, if we set $\xi_1=x$, $\xi_2=[y(t_1),
y(t_2),\dots,y(t_j)]$, and $\xi_3=y(t_{j+1})$.

To get a simple representation for the covariances in \eqref{eq:KF_m}
and \eqref{eq:KF_P}, define a new output
\begin{equation} \nonumber
  \tilde{y}:=y(t_{j+1})-\frac{t_b-t_{j+1}}{t_b-t_a}y(t_a)-\frac{t_{j+1}-t_a}{t_b-t_a}y(t_b).
\end{equation}
That is, $\tilde y$ is $y(t_{j+1})$ from which the linear interpolant
between $y(t_a)$ and $y(t_b)$ has been removed.
By plugging \eqref{eq:output} here, this can be written in the form
$\tilde y=\tilde C x + \tilde w$ where
\begin{align*}  
\tilde C &= C \int_0^{t_{j+1}}e^{As} \, ds
  -C\frac{t_b-t_{j+1}}{t_b-t_a}\int_0^{t_a}e^{As} \, ds
  -C\frac{t_{j+1}-t_a}{t_b-t_a}\int_0^{t_b}e^{As} \, ds \\
&=C \left( \frac{t_b-t_{j+1}}{t_b-t_a}
    \int_{t_a}^{t_{j+1}} e^{As} \, ds -
    \frac{t_{j+1}-t_a}{t_b-t_a}\int_{t_{j+1}}^{t_b} e^{As} \, ds \right)
\end{align*}
and 
\[
\tilde w =
w(t_{j+1})-\frac{t_b-t_{j+1}}{t_b-t_a}w(t_a)-\frac{t_{j+1}-t_a}{t_b-t_a}w(t_b).
\]
Since $w$ is Brownian motion, it holds that $\tilde w \sim N
\left(0,\frac{(t_{j+1}-t_a)(t_b-t_{j+1})}{t_b-t_a}R \right)$ and
$\tilde w$ is independent of the already included measurements (that
is, of $\xi_2$) and hence of $\tilde x_j$, as well. Thus $\Ex{\tilde
  y|\xi_2}=\tilde C \tilde x_j$,
\[
\Cov{x-\tilde x_j,\tilde y-\tilde C \tilde x_j}=P\tilde C^*,
\]
and
\[
\Cov{\tilde y-\tilde C \tilde x_j,\tilde y-\tilde C \tilde x_j}=\tilde
CP\tilde C^*+\frac{(t_{j+1}-t_a)(t_b-t_{j+1})}{t_b-t_a}R.
\]

By \eqref{eq:KF_m}, the new estimate $\tilde
x_{j+1}:=\Ex{x|\{y(t_1),y(t_2),\dots,y(t_{j+1})\}}$ is given by
\begin{equation} \label{eq:new_est} \tilde x_{j+1}=\tilde x_j+P_j
  \tilde C^* \left( \tilde C P_j \tilde
    C^*+\frac{(t_{j+1}-t_a)(t_b-t_{j+1})}{t_b-t_a}R \right)^{-1}\left( \tilde
    y - \tilde C \tilde x_j \right)
\end{equation}
and by \eqref{eq:KF_P}, the new error covariance $P_{j+1}:=\Cov{x-\tilde
  x_{j+1},x-\tilde x_{j+1}}$ by
\begin{equation} \label{eq:errCov} P_{j+1} = P_j- P_j \tilde C^* \left(
    \tilde C P_j \tilde C^*+\frac{(t_{j+1}-t_a)(t_b-t_{j+1})}{t_b-t_a}R
  \right)^{-1} \tilde C P_j.
\end{equation}

This will be used with $t_b-t_{j+1}=t_{j+1}-t_a=h$, and we define
\begin{equation} \label{eq:C_h} C_h(t)x:=\frac{C}{2}\left(
    \int_{t-h}^t e^{As}x \, ds - \int_t^{t+h} e^{As}x \, ds \right),
  \qquad \textrm{for } t \ge h >0.
\end{equation}
\begin{lemma} \label{lem:C_hbd} 
   If $C\in \mathcal{L}(\sX,\sY)$ then for $t \in [h,T-h]$ it holds that
\begin{itemize}
\item[(i)] $\ \norm{C_h(t)}_{\mathcal{L}(\sX,\sY)}\le
  h\mu\norm{C}_{\mathcal{L}(\sX,\sY)} \ $ and
\item[(ii)] $\ \norm{C_h(t)}_{\mathcal{L}(\dom(A),\sY)}\le
  \displaystyle\frac{h^2}{2} \mu\norm{A}_{\mathcal{L}(\dom(A),\sX)}
  \norm{C}_{\mathcal{L}(\sX,\sY)}$.
\end{itemize}
In the finite dimensional case $\norm{A}_{\mathcal{L}(\dom(A),\sX)}$ means
plainly the matrix norm of $A$. In the infinite dimensional case
$\norm{A}_{\mathcal{L}(\dom(A),\sX)}=1$ because $\dom(A)$ is equipped
with the graph norm of $A$.
\end{lemma}
\noindent This could also be shown for more general $\tilde C$ with
$t_b-t_a$ replacing $h$ in {\it (i)} and
$\frac{(t_{j+1}-t_a)^2}2+\frac{(t_b-t_{j+1})^2}2$ replacing $h^2$ in {\it
  (ii)} but that is not needed. Also, part {\it (ii)} can be made a
bit better. In fact, $\norm{C_h(t)x}_{\sY} \le
\frac{h^2}{2}\mu\norm{C}_{\mathcal{L}(\sX,\sY)}\norm{Ax}_{\sX}$.
\begin{proof}

Part {\it (i)} of the Lemma
is clear from the definition~\eqref{eq:C_h} since $\norm{e^{At}}_{\mathcal{L}(\sX)} \le \mu$. For part {\it (ii)}, note that $Ce^{At}x \in
C^1(\mathbb{R}^+;\sY)$ with $\frac{d}{dt}Ce^{At}x=CAe^{At}x$ and
$\norm{CAe^{At}x}_{\sY}\le
\mu\norm{C}_{\mathcal{L}(\sX,\sY)}\norm{A}_{\mathcal{L}(\dom(A),\sX)}\norm{x}_{\dom(A)}$. Then
by Bochner integral properties, $C$ can be taken inside the integral
and thus
\begin{align*}  
&\int_{t-h}^t Ce^{As}x \, ds -
   \int_t^{t+h}Ce^{As}x \, ds \\
= &\int_{t-h}^t \left(
    Ce^{At}x-\int_s^t CAe^{Ar}x \, dr \right)ds-\int_t^{t+h}
  \left(Ce^{At}x+\int_t^s CAe^{Ar}x \, dr \right)ds \\
=&-\int_{t-h}^t \int_s^t CAe^{Ar}x \, dr \, ds - \int_t^{t+h} \int_t^s CAe^{Ar}x \, dr \, ds.
\end{align*}
This together with the bound for $\norm{CAe^{At}x}_{\sY}$ imply {\it
  (ii)}.
\end{proof}

We are now ready to proceed to the actual convergence result which we shall first show in finite dimensional context, namely $\sX = \mathbb{R}^p$. The infinite dimensional generalisation will be treated below.

\begin{theorem} \label{thm:fin_dim} Let now $\sX=\mathbb{R}^p$ and $A
  \in \mathbb{R}^{p \times p}$ and $C \in \mathbb{R}^{r \times p}$
  (with $r \le p$) and let $\hat{z}_{T,n}$ and $\hat{z}(T)$ be as
  defined above in \eqref{eq:estimates}, with $u=0$ in \eqref{eq:system}. Then
\begin{equation} \nonumber
\Ex{\norm{\hat{z}_{T,n}-\hat{z}(T) }_{\sX}^2} \le \frac{MT^3}{n^2}
\end{equation}
where $M=\frac{\mu^2\tr(P_0)\Ex{\norm{\hat z_{T,n}-z(T)}_{\sX}^2}\norm{C}^2\norm{A}^2
}{12\min(\textup{eig}(R))}$.
\end{theorem}
\noindent The constant $M$ depends on $n$ through $\Ex{\norm{\hat z_{T,n}-z(T)}_{\sX}^2}$ which is the trace of the error covariance of the discrete time state estimate $\hat
z_{T,n}$. In order to get a strict \emph{a priori} result, this term can be bounded by $\Ex{\norm{\hat z_{T,n}-z(T)}_{\sX}^2} \le \mu^2 \tr(P_0)$.
\begin{proof}
   In the beginning of this section, the output \eqref{eq:output} is considered as a signal parameterized by the initial state $x$ and corrupted by noise $w$. Therefore, it seems beneficial to consider $\Ex{x|\mathcal{F}_j}$ from which estimates for $z(T)$ can be obtained through $\Ex{z(T)|\mathcal{F}_j}=e^{AT}\Ex{x|\mathcal{F}_j}$.

   The outline of the proof is as follows. First, we define the
  martingale $\tilde z_j=e^{AT}\tilde x_j$ where $\tilde x_j$ is the martingale $\tilde x_j:=\Ex{x|\mathcal{F}_j}$, where $\mathcal{F}_j=\sigma\{
  y(t),t \in \textup{T}_j \}$ and $\textup{T}_j=\{ t_i \}_{i=1}^j$ --- as explained in Section~\ref{sub:stochastics}. The time points $\{ t_i \}_{i=1}^{n}$ are $t_i=iT/n$ but  $\{ t_i \}_{i=n+1}^{\infty}$
  will be defined later.
    The martingales are Gaussian and hence square integrable, and so by Lemma~\ref{lem:tele}, we
  have the following telescope identity for $L,N \in \mathbb{N}$ with
  $L \ge N$:
\begin{equation} \label{eq:telescope}
\Ex{\norm{\tilde z_L-\tilde
        z_N}_{\sX}^2}=\sum_{j=N}^{L-1} \Ex{\norm{\tilde z_{j+1}-\tilde
        z_j}_{\sX}^2}.
\end{equation}
Second, we find an upper bound for $\Ex{\norm{\tilde z_{j+1}-\tilde
    z_j}_{\sX}^2}$ using the results of
Section~\ref{sub:Kalman} and the beginning of this section. Third, we prove that the sum in
\eqref{eq:telescope} converges as $L \to \infty$ and thus $\tilde z_j$
is a Cauchy sequence in $L^2(\Omega;\sX)$. It has a limit in this
space by completeness and the limit must be $\hat z(T)$ by the
considerations in Section~\ref{sub:stochastics}. Also, setting $N=n$
(we have $\tilde z_n=\hat z_{T,n}$) and letting $L \to \infty$ in
\eqref{eq:telescope} gives $\Ex{\norm{\hat{z}_{T,n}-\hat{z}(T) }_{\sX}^2}$.

{\bf (I) Martingale $\tilde z_j$:} Let $t_i=iT/n$ for
$i=1,\dots,n$. Then $\tilde z_j$ for $j=1,\dots,n$ are the state estimates
from the discrete time Kalman filter and, in particular, $\tilde
z_n=\hat z_{T,n}$  defined in \eqref{eq:estimates}.  The idea is to
then halve the intervals $((l-1)T/n,lT/n)$ for $l=1,\dots,n$ between the
already included measurements. That is, we include measurements $y(t_i)$ where $i=n+1,...,2n$, and $t_i=\frac{(i-n-1/2)T}{n}$.
Then we halve the
new intervals $((l-1)T/2n,lT/2n)$ for $l=1,\dots,2n$ by including $2n$
measurements $y(t_i)$ for $i=2n+1,...,4n$ and $t_i= \frac{(i-2n-1/2)T}{2n}$
and so on. This addition of new time points is illustrated in Fig.~1.

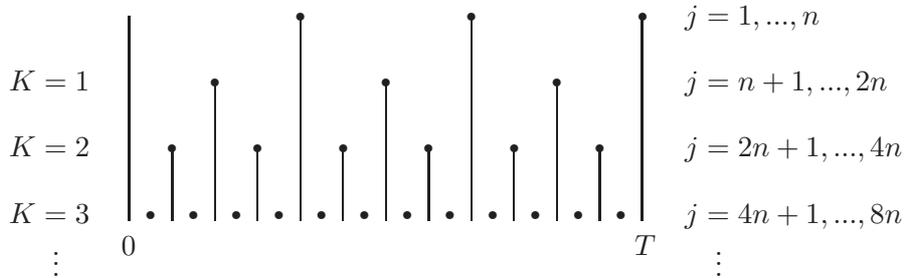
\begin{figure}[b]
\center
\begin{picture}(270,105)
\thicklines
\put(10,100){\line(0,-1){77}}
\put(202,100){\line(0,-1){77}}

\thinlines
\put(74,100){\circle*{3}}
\put(74,100){\line(0,-1){77}}
\put(138,100){\circle*{3}}
\put(138,100){\line(0,-1){77}}
\put(202,100){\circle*{3}}

\put(42,75){\circle*{3}}
\put(42,75){\line(0,-1){52}}
\put(106,75){\circle*{3}}
\put(106,75){\line(0,-1){52}}
\put(170,75){\circle*{3}}
\put(170,75){\line(0,-1){52}}

\put(26,50){\circle*{3}}
\put(26,50){\line(0,-1){27}}
\put(58,50){\circle*{3}}
\put(58,50){\line(0,-1){27}}
\put(90,50){\circle*{3}}
\put(90,50){\line(0,-1){27}}
\put(122,50){\circle*{3}}
\put(122,50){\line(0,-1){27}}
\put(154,50){\circle*{3}}
\put(154,50){\line(0,-1){27}}
\put(186,50){\circle*{3}}
\put(186,50){\line(0,-1){27}}

\put(18,25){\circle*{3}}
\put(34,25){\circle*{3}}
\put(50,25){\circle*{3}}
\put(66,25){\circle*{3}}
\put(82,25){\circle*{3}}
\put(98,25){\circle*{3}}
\put(114,25){\circle*{3}}
\put(130,25){\circle*{3}}
\put(146,25){\circle*{3}}
\put(162,25){\circle*{3}}
\put(178,25){\circle*{3}}
\put(194,25){\circle*{3}}

\put(218,97){$j = 1,...,n$}
\put(218,72){$j=n+1,...,2n$}
\put(218,47){$j=2n+1,...,4n$}
\put(218,22){$j=4n+1,...,8n$}

\put(-35,72){$K=1$}
\put(-35,47){$K=2$}
\put(-35,22){$K=3$}

\put(7,10){0}
\put(199,10){$T$}
\put(229,2){$\vdots$}
\put(-19,2){$\vdots$}

\label{fig:times}
\end{picture}
\caption{Illustration of the time point addition scheme in the construction of the martingales $\tilde x_j$ and $\tilde z_j$.}
\end{figure}

{\bf (II) Increment $\tilde z_{j+1}-\tilde z_j$:} Assume that the
current state estimate is $\tilde{z}_j=e^{AT}\tilde{x}_j$ with $j \ge n$, the
corresponding error covariance matrices are $e^{AT}P_je^{A^*T}$ and $P_j$, respectively, and the next
measurement being included is $y(t_{j+1})$ with $j+1 \in \{ 2^{K-1}n+1,...,2^Kn \}$ for some $K =1,2,...$ (see Fig.~1). Then $t_{j+1}=(2(j-2^{K-1}n)+1)h$ with $h=\frac{T}{2^Kn}$.
The new initial state estimate $\tilde x_{j+1}$ is then
given by \eqref{eq:new_est} with $\tilde
C=C_h\left((2(j-2^{K-1}n)+1)h \right)$ --- denoted below simply
by $C_h$ --- and $h=\frac{T}{2^Kn}$. We are only interested in the
$L^2(\Omega;\sX)$-norm of the $\tilde z$-process increment, and as discussed in
Section~\ref{sub:Kalman}, it is obtained from the covariance increment
given in~\eqref{eq:errCov}:
\[
\Ex{\norm{\tilde z_{j+1}-\tilde z_j}_{\sX}^2}=\tr\left(e^{AT}P_jC_h^*(C_hP_jC_h^*+h/2 \,
  R)^{-1}C_hP_je^{A^*T} \right).
\]
Now we wish to establish a bound for this trace.  To this end, recall
that the norm of the inverse of a positive definite matrix is $
\norm{Q^{-1}}=\frac1{\min(\textup{eig}(Q))}$, and thus,
\begin{equation} \label{eq:C_R} \norm{ \left( C_h P_j C_h^*+\frac{h}{2}R
    \right)^{-1}} \le \frac{2}{h\min(\textup{eig}(R))}=:\frac{C_R}{h}.
  \end{equation}
  Using this and part {\it (ii)} of Lemma~\ref{lem:C_hbd} gives
\begin{align} \nonumber
  & \hspace{-14mm}\tr \! \left( \! e^{AT}P_j C_h^* \left( C_h P_j
      C_h^*+\frac{h}{2}R \right)^{\! -1} \!\! C_h P_j e^{A^*T}\right) \\ \nonumber  &=\sum_{k=1}^p 
  \ip{C_hP_je^{A^*T}e_k,\left( C_h P_j C_h^*+\frac{h}{2}R \right)^{\! -1} \!\! C_h P_je^{A^*T}e_k} \\
  \nonumber
  \hspace{20mm}& \le \frac{C_R}{h}\sum_{k=1}^p \norm{C_hP_je^{A^*T}e_k}_{\sY}^2=\frac{C_R}{h}\sum_{k=1}^p \norm{\Ex{C_h(\tilde x_j-x)\ip{e^{AT}(\tilde x_j-x),e_k}_{\sX}}}_{\sY}^2 
\\
  \nonumber & \le \frac{C_R}{h}\Ex{\norm{C_h(\tilde
      x_j-x)}_{\sY}^2} \sum_{k=1}^p \Ex{\ip{e^{AT}(\tilde x_j-x),e_k}_{\sX}^2} \\  \label{eq:inter} &\le \frac{C_R}{h} \tr(C_hP_jC_h^*)\Ex{\norm{\tilde z_j-z(T)}_{\sX}^2}
\\
\label{eq:bdbd}
& \le \frac{h^3}{2\min(\textup{eig}(R))}\mu^2\norm{C}^2\norm{A}^2\tr(P_j)\Ex{\norm{\tilde z_j-z(T)}_{\sX}^2}.
\end{align}

{\bf (III) Convergence:} It holds that 
$\tr(P_j) \le \tr(P_0)$ and $\Ex{\norm{\tilde z_j-z(T)}_{\sX}^2} \le \Ex{\norm{\hat z_{T,n}-z(T)}_{\sX}^2}$.  In part (II) of the proof we had
$h=2^{-K}T/n$ and that bound is used for all $2^{K-1}n$ new
measurements corresponding to this $h$.  Finally, setting $N=n$ and $L
\to \infty$ in \eqref{eq:telescope} and using \eqref{eq:bdbd} to bound
the terms of the sum yields
\begin{align*} 
  \Ex{\norm{\hat{z}_{T,n}-\hat{z}(T)}_{\sX}^2} &\le \sum_{K=1}^{\infty}
  2^{K-1}n \left( \frac{T}{2^Kn} \right)^{\!
    3}\frac{\mu^2\tr(P_0)\Ex{\norm{\hat z_{T,n}-z(T)}_{\sX}^2}\norm{C}^2\norm{A}^2}{2\min(\textup{eig}(R))}
  \\ &=\frac{\mu^2\tr(P_0)\Ex{\norm{\hat z_{T,n}-z(T)}_{\sX}^2}\norm{C}^2\norm{A}^2
    T^3}{12\min(\textup{eig}(R))n^2}
\end{align*}
completing the proof.
\end{proof}

\section{Input noise} \label{sec:noise}

The case with input noise follows exactly the same outline as the case without input noise. The solution to \eqref{eq:system} is given by the Wiener integral
\[
z(t)=e^{At}x+\int_0^t e^{A(t-s)}Bdu(s).
\]
The idea now is to consider the output $y(t)$ as a process that is parameterized by the initial state $x$ and the input noise process $u$. To this end, let us define the solution operator $S(t)$ through
\begin{equation} \label{eq:Sdef}
S(t):[x,u] \mapsto e^{At}x+\int_0^t e^{A(t-s)}Bdu(s)
\end{equation}
and so $z(t)=S(t)[x,u]$. Then the state estimate over a given sigma algebra $\sigma$ is given by
\[
\Ex{z(T)|\sigma}=S(T)\Ex{[x,u] | \sigma}.
\]
Hence, we shall virtually consider the estimate of the combined initial state $x$ and the noise process $u$ and then the actual state estimate is obtained through $S(T)$.
\begin{theorem} \label{thm:noise}
Let now $\sX=\mathbb{R}^p$ and $A
  \in \mathbb{R}^{p \times p}$ and $C \in \mathbb{R}^{r \times p}$
  (with $r \le p$) and let $\hat z_{T,n}$ and $\hat z(T)$ be as defined in \eqref{eq:estimates}. Then \\
\[
\Ex{\norm{\hat z_{T,n}-\hat z(T)}_{\sX}^2 }\le \frac{M_1 T^2}n + \frac{M_2 T^3}{n^2} + \frac{M_3 T^4}{n^2}
\] 
 where $M_1=\frac{\norm{C}^2 \tr(BQB^*)\Ex{\norm{\hat z_{T,n}-z(T)}_{\sX}^2}}{\min(\textup{eig}(R))}$, $M_2=\frac{\mu^2\norm{A}^2 \norm{C}^2 \tr(P_0)\Ex{\norm{\hat z_{T,n}-z(T)}_{\sX}^2}}{12\min(\textup{eig}(R))} $, and  $M_3=\frac{\mu^2\norm{C}^2\tr(ABQB^* \! A^*)\Ex{\norm{\hat z_{T,n}-z(T)}_{\sX}^2}}{2\min(\textup{eig}(R))}$.
\end{theorem}
As in Theorem~\ref{thm:fin_dim}, an \emph{a priori} result is obtained by bounding \linebreak $\Ex{\norm{\hat z_{T,n}-z(T)}_{\sX}^2} \le \mu^2 \tr(P_0)+T\mu^2\tr(BQB^*)$.

In the bound of this theorem, the second term originates from the error in the initial state. From the proof below (after \eqref{eq:bd_inc}) it can be seen that in fact, the different error components can be treated separately (compare \eqref{eq:bd_inc} with \eqref{eq:inter}).
\begin{proof}
The proof follows exactly the same outline as the proof of Theorem~\ref{thm:fin_dim}.
Say we are estimating $[x,u]$ and we have $\tilde x_j := \Ex{[x,u]|\mathcal{F}_j}$ and the corresponding error covariance $\mathbb{P}_j$. Then the state estimate and the corresponding error covariance are given by $\tilde z_j = S(T)\tilde x_j$ and $S(T)\mathbb{P}_jS(T)^*$ --- although in the last equation the formal adjoint $S(T)^*$ is only defined in connection  with the covariance $\mathbb{P}_j$, namely \newline $\mathbb{P}_jS(T)^*h=\Ex{(\tilde x_j-[x,u])\ip{S(T)(\tilde x_j-[x,u]),h}_{\sX}}$.

 Say we are including measurement $y(t_{j+1})$ which can be written as
\[
y(t_{j+1})=C\int_0^{t_{j+1}}e^{As}x \, ds + C \int_0^{t_{j+1}} \!\! \int_0^s e^{A(s-r)}B \, du(r) \, ds + w(t_{j+1}).
\]
As in the previous section, to get an output with output noise that is uncorrelated with the already included outputs, we shall subtract the linear interpolant from $y(t_{j+1})$, namely define
\[
\tilde y = y(t_{j+1})-\frac12y(t_a)-\frac12y(t_b)
\]
where $t_a=t_{j+1}-h$ and $t_b=t_{j+1}+h$ for proper $h$. Now this output can be written as
\[
\tilde y = \Ch[x,u]^T+\tilde w
\]
where $\tilde w \sim N(0,h/2 \, R)$ and $\Ch=[C_h(t_{j+1}), C_{h,u}(t_{j+1})]$ with $C_h(t_{j+1})$ defined in \eqref{eq:C_h} and
\begin{equation} \label{eq:Chu}
C_{h,u}(t)u:=\frac{C}{2}\left( \int_{t-h}^t \int_0^s e^{A(s-r)}Bdu(r) \, ds - \int_t^{t+h} \!\! \int_0^s e^{A(s-r)}Bdu(r) \, ds \right)
\end{equation}
for $t \ge h$. Now the error covariance increment is as before in \eqref{eq:KF_P} and \eqref{eq:errCov} (but with this new output operator $\Ch$) and the $L^2(\Omega;\sX)$-norm increment $\Ex{\norm{\tilde z_{j+1}-\tilde z_j}_{\sX}^2 }$ is  given by
\begin{align}
\Ex{\norm{\tilde z_{j+1}-\tilde z_j}_{\sX}^2 }&=\tr \left(S(T)\mathbb{P}_j\Ch^*\left(\Ch \mathbb{P}_j \Ch^*+\frac{h}{2}R\right)^{-1}\Ch \mathbb{P}_jS(T)^*\right) \nonumber \\
&=\sum_{k=1}^p \ip{\Ch \mathbb{P}_j S(T)^*e_k,\left(\Ch \mathbb{P}_j \Ch^*+\frac{h}{2}R\right)^{-1}\Ch \mathbb{P}_jS(T)^*e_k}_{\!\! \sY} \nonumber \\
&\le \frac{C_R}{h}\sum_{k=1}^p \norm{\Ch \mathbb{P}_j S(T)^* e_k}_{\sY}^2 \nonumber \\
&=\frac{C_R}{h}\sum_{k=1}^p \norm{\Ex{\Ch(\tilde x_j-[x,u])\ip{S(T)(\tilde x_j-[x,u]),e_k}_{\sX} }}_{\sY}^2 \nonumber \\
& \le \frac{C_R}{h} \Ex{\norm{\Ch (\tilde x_j-[x,u])}_{\sY}^2} \sum_{k=1}^p \Ex{\ip{S(T)(\tilde x_j-[x,u]),e_k}_{\sX}^2} \nonumber \\
& \le \frac{C_R}{h} \Ex{\norm{\Ch [x,u]}_{\sY}^2} \Ex{\norm{\hat z_{T,n}-z(T)}_{\sX}^2}. \label{eq:bd_inc}
\end{align}

In order to get a suitable bound for the increment, we must find a bound for the term $\Ex{\norm{\Ch [x,u]}_{\sY}^2} \le  \Ex{\norm{C_h x}_{\sY}^2}+\Ex{\norm{C_{h,u} u}_{\sY}^2}$ (recall that $x$ and $u$ are independent). As in the proof of Theorem~\ref{thm:fin_dim}, by Lemma~\ref{lem:C_hbd}, the first part is bounded by
\begin{equation} \label{eq:bd_Chx}
\Ex{\norm{C_h x}_{\sY}^2} \le \frac{h^4}{4} \mu^2\norm{A}^2 \norm{C}^2 \tr(P_0)
\end{equation}
so then remains the input noise induced term. To evaluate $C_{h,u}(t_{j+1})u$, note that
\[
\int_0^t e^{A(t-s)}Bdu(s)=\int_0^{t_{j+1}} \! e^{A(t-s)}Bdu(s) +\int_{t_{j+1}}^t \! A \int_0^s e^{A(s-r)}Bdu(r) \, ds + \int_{t_{j+1}}^t B du(s)
\]
for $t \ge t_{j+1}$ --- for $t < t_{j+1}$, just change $t_{j+1} \leftrightarrow t$ in the bounds of the last two integrals and put minus signs in front of them. Of course the last term is just $\int_{t_{j+1}}^t B du(s)=B(u(t)-u(t_{j+1}))$. Applying this to \eqref{eq:Chu} gives
\begin{align}
\nonumber
C_{h,u}(t_{j+1})u=&-\frac{C}{2} \Bigg[  \int_{t_{j+1}-h}^{t_{j+1}} \left( \int_t^{t_{j+1}} \!\! A \int_0^s e^{A(s-r)}Bdu(r) \, ds + B\big(u(t)-u(t_{j+1})\big) \right) dt \\ \label{eq:Chu_mod}
& \qquad + \int_{t_{j+1}}^{t_{j+1}+h} \left( \int_{t_{j+1}}^t \!\! A \int_0^s e^{A(s-r)}Bdu(r) \, ds + B\big(u(t)-u(t_{j+1})\big) \right) dt \Bigg]. 
\end{align}
These two terms are very similar by nature so it suffices to find a bound for one of them and 
use the same bound for both terms. Thus, let us consider the first part of the latter term, namely
\begin{align*}
&\frac{C}{2}\int_{t_{j+1}}^{t_{j+1}+h} \!\! \int_{t_{j+1}}^t \!\! A \int_0^s e^{A(s-r)}Bdu(r) \, ds \, dt \\ & =\frac{C}{2}\int_{t_{j+1}}^{t_{j+1}+h} \!\! \int_{t_{j+1}}^t \!\! A \int_0^{t_{j+1}} e^{A(s-r)}Bdu(r) \, ds \, dt+\frac{C}{2}\int_{t_{j+1}}^{t_{j+1}+h} \!\! \int_{t_{j+1}}^t \!\! A \int_{t_{j+1}}^s e^{A(s-r)}Bdu(r) \, ds \, dt \\
& = \frac{C}{2}\int_0^{t_{j+1}} \!\! \int_{t_{j+1}}^{t_{j+1}+h} \!\! (t_{j+1}+h-s)Ae^{A(s-r)}Bds \, du(r) \\ & \quad + \frac{C}{2}\int_{t_{j+1}}^{t_{j+1}+h} \!\! \int_r^{t_{j+1}+h} \!\! ({t_{j+1}}+h-s)Ae^{A(s-r)}Bds \, du(r) \\
&=(I) + (II).
\end{align*}
Then
\begin{align*}
\Cov{(I),(I)}=\frac14\int_0^{t_{j+1}} \!\! \int_{t_{j+1}}^{t_{j+1}+h} \!\! \int_{t_{j+1}}^{t_{j+1}+h} \!\!\! & (t_{j+1}+h-s)(t_{j+1}+h-r) \\ & Ce^{A(s-t)}ABQB^*A^*e^{A^*(r-t)}C^*dr \, ds \, dt
\end{align*}
and from this, using the bound $\norm{e^{At}}_{\mathcal{L}(\sX)} \le \mu$,
\[
\Ex{\norm{(I)}_{\sY}^2}=\tr\big(\Cov{(I),(I)}\big) \le \frac{t_{j+1}h^4}8\mu^2 \norm{C}^2\tr(ABQB^*A^*).
\]
For the second term we have
\begin{align*}
\Cov{(II),(II)}=\frac14 \int_{t_{j+1}}^{t_{j+1}+h} \!\! \int_t^{t_{j+1}+h} \!\! \int_t^{t_{j+1}+h} \!\!\! & (t_{j+1}+h-s)(t_{j+1}+h-r) \\ & CAe^{A(t-s)}ABQB^*A^*e^{A^*(t-r)}C^*dr \, ds \, dt
\end{align*}
and, again,
\[
\Ex{\norm{(II)}_{\sY}^2}=\tr\big(\Cov{(II),(II)}\big) \le \frac{h^5}8 \mu^2 \norm{C}^2\tr(ABQB^*A^*).
\]
In $(I)$, $r \in [0,t_{j+1}]$ and in $(II)$, $r \in [t_{j+1},t_{j+1}+h]$, and thus they are independent.
Using this and the fact that $t_{j+1}+h \le T$, gives 
\[
\Ex{\norm{(I)+(II)}_{\sY}^2} \le \frac{Th^4}8 \mu^2 \norm{C}^2\tr(ABQB^*A^*).
\]
It is well known that
\[
\Cov{ \int_0^h Bu(t) \,dt, \int_0^h Bu(t) \,dt}=\frac{h^3}3 BQB^*
\]
and so
\[
\Ex{\norm{\frac{C}2\int_{t_{j+1}}^{t_{j+1}+h} B(u(t)-u(t_{j+1})) dt}_{\sY}^2} \le \frac{h^3}{12}\norm{C}^2 \tr(BQB^*).
\]
In \eqref{eq:Chu_mod}, the two $B(u(t)-u(t_{j+1}))$-terms are independent (because in the first one, $t \le t_{j+1}$ and in the second, $t \ge t_{j+1}$) and by utilizing this and gathering the above bounds, we get
\begin{align*}
\Ex{\norm{C_{h,u}(t_{j+1})u}_{\sY}^2} &\le 6\Ex{\norm{(I)+(II)}_{\sY}^2}+6  \frac{h^3}{12}\norm{C}^2\tr(BQB^*) \\
& \le \frac{3Th^4}4 \mu^2 \norm{C}^2\tr(ABQB^*A^*)+\frac{h^3}{2}\norm{C}^2 \tr(BQB^*).
\end{align*}
Combining this with \eqref{eq:bd_inc} and \eqref{eq:bd_Chx} gives
\begin{align*}
\Ex{\norm{\tilde z_{j+1}-\tilde z_j}_{\sX}^2 } \le&
 \frac{h^3}{4}\mu^2 C_R \norm{A}^2 \norm{C}^2 \tr(P_0)\Ex{\norm{\hat z_{T,n}-z(T)}_{\sX}^2} \\ & +\frac{3Th^3}4 \mu^2C_R\norm{C}^2\tr(ABQB^*A^*)\Ex{\norm{\hat z_{T,n}-z(T)}_{\sX}^2} \\ & +\frac{h^2}{2}C_R\norm{C}^2 \tr(BQB^*)\Ex{\norm{\hat z_{T,n}-z(T)}_{\sX}^2} \\
=: & C_3 h^3+C_2h^2.
\end{align*}
Using this bound as in the end of the proof of Theorem~\ref{thm:fin_dim} gives
\[
\Ex{\norm{\hat z_{T,n}-\hat z(T)}_{\sX}^2 } \le \frac{C_2T^2}{2n}+\frac{C_3T^3}{6n^2}
\]
completing the proof.
\end{proof}

\section{Generalization to infinite dimensional systems} \label{sec:infdim}

We move on to infinite dimensional state space $\sX$. Compared to
the finite dimensional case, the main difficulty arises from that the
bound for $C_h$ in part {\it (ii)} of Lemma~\ref{lem:C_hbd} utilizes
the differentiability of $Ce^{At}x$ and thus it holds for $x \in
\dom(A)$. A natural assumption that would make it possible to use this
bound is that $x$ is a $\dom(A)$-valued random variable. This is
exactly what is done in Theorem~\ref{thm:bddC_nat}. Before that, in
Theorem~\ref{thm:conv_bddC} we shall see, however, that a reasonable
convergence estimate can be obtained with slightly less smooth initial
state $x$. Before tackling this problem, we present an example
illuminating the necessity of some additional assumptions.

\begin{example} \label{ex:wave} {\rm This example shows that there is
    a system with $C \in \mathcal{L}(\sX,\mathbb{R})$ such that
    $\Ex{\norm{\hat{z}_{T,n} - \hat{z}(T)}_{\sX}^2}$ converges
    arbitrarily slowly where $\hat{z}_{T,n}$ and $\hat{z}(T)$ are
    defined in \eqref{eq:estimates}.  Consider the one-dimensional
    wave equation (without input noise) with augmented state
    vector, \begin{equation} \label{eq:wave} \begin{cases}
        \dfrac{d}{dt} \! \bbm{z_1(s,t) \\ z_2(s,t)} = \bbm{0 & I \\
          \frac{\partial^2}{\partial
            s^2} & 0}\bbm{z_1(s,t) \\ z_2(s,t)}, \quad s \in [0,1], \ t \in \mathbb{R}^+, \\
        z_1(s,0)=0, \ z_2(s,0) = x(s), \\
        dy(t)=Cz(t) \, dt + dw(t) \end{cases} \end{equation} in state
    space $\sX=H_0^1[0,1] \times L^2(0,1)$ and $\dom (A)=
    \left(H^2[0,1] \cap H_0^1[0,1]\right)\times H_0^1[0,1]$. The state is $z(t)=[z_1(t) \ z_2(t)]^T$.
    The
    output operator $C \in \mathcal{L}(\sX,\mathbb{R})$ is given by
    $Cz = \int_0^1 c(s)z_1(s) \, ds$ where $c(s)=\sumk c_k e_k(s)$ with
    some $\{c_k \} \in l^2$ and $\{e_k\}$ is the orthonormal basis in
    $L^2(0,1)$ formed by the sine functions, that is
    $e_k(s)=\frac1{\sqrt{2}}\sin(k \pi s)$.  The initial velocity is
    $x= \sumk a_k e_{2^k}$ where $a_k \sim N(0,\sigma_k^2)$ and $a_k
    \perp a_i$ for $k \ne i$. It holds that
    $\Ex{\norm{x}_{\sX}^2}=\sumk \sigma_k^2$ and thus this sum is
    assumed to converge. Then the solution to \eqref{eq:wave} and the
    corresponding output are \begin{equation} \nonumber \begin{cases}
        z_1(s,t)= \frac1{\sqrt{2}} \sumk a_k \sin(2^k \pi s) \sin(2^k \pi t), \\
        z_2(s,t)= \frac1{\sqrt{2}} \sumk a_k \sin(2^k \pi s) \cos(2^k \pi t), \\
        dy(t)=\frac1{\sqrt{2}} \sumk a_k c_{2^k} \sin(2^k \pi t) \, dt
        + dw(t).  \end{cases}
    \end{equation}

    Now set $T=1$ and consider the subsequence $\hat
    z_{T,2^l}$ of the discrete time estimates, defined in \eqref{eq:estimates}. As noted in the proof of
    Thm.~\ref{thm:fin_dim}, it holds that $\Ex{\norm{\hat
        z_{T,2^l}-\hat z(T)}_{\sX}^2}=\sum_{i=l}^{\infty}
    \Ex{\norm{\hat z_{T,2^{i+1}}-\hat z_{T,2^i}}_{\sX}^2}$. The
    estimate $\hat z_{T,2^{l+1}}$ is obtained from the previous
    estimate $\hat z_{T,2^l}$ by including measurements $y \left(
      \frac{2i-1}{2^{l+1}}\right)$ for $i=1,\dots,2^l$ as described in the beginning of
     Section~\ref{sec:noiseless}. 
In order to obtain a lower bound for
    $\Ex{\norm{\hat z_{T,2^{l+1}}-\hat z_{T,2^l}}_{\sX}^2}$, define \linebreak
    $\widehat C:=[C_h(h),C_h(3h),\dots,C_h(1-h)]^T:\sX \to
    \mathbb{R}^{2^l}$ where $h=\frac1{2^{l+1}}$. That is, $\widehat C$
    gives the whole batch of the measurements needed for the
    update. For the wave equation it holds that $\norm{z(t)}_{\sX}=\norm{x}_{\sX}$ and so the increments $\Ex{\norm{\hat z_{T,2^{l+1}}-\hat z_{T,2^l}}_{\sX}^2}$ are the same as the corresponding increments for $\tilde x_{2^l} = \Ex{x| \{ y(t), \, t=j/2^l, \, j=1,...,2^l\}}$.
        Then denoting $P_l=\Cov{\tilde x_{2^l}-x,\tilde
      x_{2^l}-x}$, it holds that 
\begin{align*} 
      &\Ex{\norm{\hat z_{T,2^{l+1}}-\hat
          z_{T,2^l}}_{\sX}^2}=\tr\left(P_l\widehat C^* \left( \widehat
          CP_l \widehat C^*+\frac{h}{2}RI \right)^{-1} \! \widehat C
        P_l \right) \\
&\ge
      \ip{\widehat C P_l e_{2^{l+1}},\left( \widehat CP_l \widehat C^*
          \! +\frac{h}{2}RI \right)^{\! -1} \!\! \widehat C P_l
        e_{2^{l+1}} \! }_{\!\! \mathbb{R}^{2^l}}
      \! \ge \frac{\norm{\widehat C P_l
          e_{2^{l+1}}}_{\mathbb{R}^{2^l}}^2}{\max \! \left( \textup{eig} \! \left(\widehat CP_l \widehat
            C^* \! +\frac{h}{2}RI \right) \! \right)}.
\end{align*}

For $h=2^{-l}$ it holds that $C_h(ih)e_{2^k}=0$ when $l<k$ and
$i=1,\dots,2^l-1$ because when computing $C_h(ih)e_{2^k}$ by
\eqref{eq:C_h}, the integrals are always over full periods of the sine
function $\sin(2^k \pi t)$. When $l=k$ it holds that
$C_h(ih)e_{2^k}=\frac{\sqrt{2}h}{\pi}c_{2^k}$ for every
$i=1,3,\dots,2^k-1$. So, loosely speaking, the already included output
values $y\left( \frac{2i-1}{2^l} \right)$ do not carry any information
on $a_k$ for $k>l$. Thus $P_le_{2^{l+1}}=\sigma_{l+1}^2e_{2^{l+1}}$
and $\norm{\widehat C P_l e_{2^{l+1}}}_{\mathbb{R}^{2^l}}^2=2^l
\sigma_{l+1}^2 \left( \frac{\sqrt{2}h}{\pi}c_{2^{l+1}} \right)^2$. For
the denominator it holds by part {\it (i)} of Lemma~\ref{lem:C_hbd}
that
\begin{equation} \nonumber \max \! \left( \! \textup{eig} \!
    \left(\widehat CP_l \widehat C^* \! +\frac{h}{2}RI \right) \!
  \right) \le \frac{h}{2}R+\Ex{\norm{\widehat C
      x}_{\mathbb{R}^{2^l}}^2 \! } \le \frac{h}{2}R+2^l h^2
  \norm{C}_{\mathcal{L}(\sX,\mathbb{R})}^2 \tr(P_0).
\end{equation}
Recalling $h=\frac1{2^{l+1}}$, we finally get
$\Ex{\! \norm{\hat z_{T,2^{l+1}} \! -\hat
      z_{T,2^l} \! }_{\sX}^2 \!} \ge
  \frac{8\sigma_{l+1}^2c_{2^{l+1}}^2}{\pi^2R+4\pi^2
    \norm{C}_{\! \mathcal{L}(\sX,\mathbb{R})}^2 \tr(P_0)}$
and further
\begin{equation} \nonumber 
\Ex{\norm{\hat z_{T,2^l}-\hat
      z(T)}_{\sX}^2} \ge
  \frac{8\sum_{i=l+1}^{\infty} \sigma_i^2c_{2^i}^2}{\pi^2R+4\pi^2
    \norm{C}_{\mathcal{L}(\sX,\mathbb{R})}^2 \tr(P_0)}
\end{equation}
where there is no $h$-dependence and the variances $\{\sigma_k^2 \}$
can be chosen so that the convergence is arbitrarily slow, concluding
the example.
}
\end{example}

Clearly some additional assumptions are needed for getting any
convergence rate estimates. In the following theorem, the initial state
is assumed to be so smooth that the covariance operator satisfies $P_0
\in \mathcal{L}(\sX,\dom(A))$. As noted after Theorem~\ref{thm:fin_dim}, the error components stemming from the initial state and the input noise can be treated separately. Therefore, the following two theorems treat the noiseless case and the input noise is treated in Corollary~\ref{cor:infdim_noise}.
\begin{theorem} \label{thm:conv_bddC} Let $\hat{z}_{T,n}$ and
  $\hat{z}(T)$ be as defined in \eqref{eq:estimates} with $u=0$ in \eqref{eq:system}, and assume $C
  \in \mathcal{L}(\sX,\sY)$. Assume $x \sim N(m,P_0)$ where the covariance
  operator satisfies $P_0 \in \mathcal{L}(\sX,\dom(A))$. Then 
  \begin{equation} \nonumber
    \Ex{\norm{\hat{z}_{T,n}-\hat{z}(T)}_{\sX}^2} \le
    \frac{MT^2}{n}
\end{equation}
where $M=\frac{r \,
  \norm{P_0}_{\mathcal{L}(\sX,\dom(A))}\norm{C}_{\mathcal{L}(\sX,\sY)}^2\Ex{\norm{\hat z_{T,n}-z(T)}_{\sX}^2}}{2\min(\textup{eig}(R))}$. Recall that \newline  $\Ex{\norm{\hat z_{T,n}-z(T)}_{\sX}^2}\le \mu^2\tr(P_0)$.
\end{theorem}
\begin{proof}
   The main idea of the proof is the same as in the proof of
  Theorem~\ref{thm:fin_dim} and we note that every step taken until
  equation \eqref{eq:inter} in that proof can be taken in the infinite
  dimensional setting as well --- $p$ just has to be replaced by
  $\infty$ in the sums but this does not cause any problems.

  So we pick up from \eqref{eq:inter} and note first that
  \begin{align*} & \tr \left( C_hP_jC_h^* \right) \le
    r\norm{C_hP_jC_h^*}_{\mathcal{L}(\sY)}=r \! \sup_{\norm{y}_{\sY}=1}\ip{y,C_hP_jC_h^*y}_{\sY} \\
= &r \!
  \sup_{\norm{y}_{\sY}=1}\ip{C_h^*y,P_jC_h^*y}_{\sX} \le r \!
  \sup_{\norm{y}_{\sY}=1}\ip{C_h^*y,P_0C_h^*y}_{\sX}=r\norm{C_hP_0C_h^*}_{\mathcal{L}(\sY)}
\end{align*}
where $r=\dim(\sY)$.  The inequality $P_j \le P_0$ was used in $\sX$,
but now the $\mathcal{L}(\sX,\dom(A))$-norm can be used for
$P_0$. Then using both parts {\it (i)} and {\it (ii)} of
Lemma~\ref{lem:C_hbd} gives
\begin{equation} \nonumber \norm{C_hP_0C_h^*}_{\mathcal{L}(\sY)} \le
  \frac{h^3}2 \mu^2\norm{C}_{\mathcal{L}(\sX,\sY)}^2\norm{P_0}_{\mathcal{L}(\sX,\dom(A))}.
\end{equation}

As before, this leads to an estimate
\begin{equation} \nonumber \Ex{\norm{\hat{z}_{T,n}-\hat{z}(T)}_{\sX}^2}
  \le \frac{r \mu^2 
    \norm{P_0}_{\mathcal{L}(\sX,\dom(A))}\norm{C}_{\mathcal{L}(\sX,\sY)}^2
    \Ex{\norm{\hat z_{T,n}-z(T)}_{\sX}^2}T^2}{2\min(\textrm{eig}(R)) \, n} =: \frac{MT^2}n
\end{equation}
completing the proof.
\end{proof}

\noindent Checking the assumption $P_0 \in \mathcal{L}(\sX,\dom(A))$
might be difficult. Under the stronger smoothness assumption $x \in
\dom(A)$ almost surely, we get the same convergence rate as in the
finite dimensional case:
\begin{theorem} \label{thm:bddC_nat} Make the same assumptions as in
  Theorem~\ref{thm:conv_bddC}. Assume, in addition, that $x \in
  \dom(A)$ almost surely. Then
\begin{equation} \nonumber 
    \Ex{\norm{\hat{z}_{T,n}-\hat{z}(T)}_{\sX}^2} \le
    \frac{MT^3}{n^2}
\end{equation}
where
$M=\frac{\mu^2\tr(AP_0A^*)\norm{C}_{\mathcal{L}(\sX,\sY)}^2\Ex{\norm{\hat z_{T,n}-z(T)}_{\sX}^2}
}{12\min(\textup{eig}(R))}$.
\end{theorem}
\begin{proof}
  The proof is the same as that of Theorem~\ref{thm:fin_dim} but from
  Eq.  \eqref{eq:inter} we proceed differently. It holds that
  \begin{align*} 
\tr\left(C_hP_jC_h^* \right) &\le
    \tr\left(C_hP_0C_h^* \right) = \Ex{\norm{C_hx}_{\sY}^2}
    \le \frac{h^4}{4}\mu^2\norm{C}_{\mathcal{L}(\sX,\sY)}^2 \Ex{\norm{Ax}_{\sX}^2}
\end{align*}
where the last inequality holds by part {\it (ii)} of
Lemma~\ref{lem:C_hbd}. The term is finite by
Proposition~\ref{prop:X_1} and Fernique's theorem. Further, it holds
that $\Ex{\norm{Ax}_{\sX}^2}=\tr(AP_0A^*)$.  Now the
result follows as above.
\end{proof}
As discussed after Theorem \ref{thm:noise}, the error components stemming from the initial state error and the input noise can be treated separately. Therefore, as an almost direct corollary of Theorems \ref{thm:noise},\ref{thm:conv_bddC}, and \ref{thm:bddC_nat}, we obtain the following result:
\begin{corollary} \label{cor:infdim_noise}
Let $\hat{z}_{T,n}$ and  $\hat{z}(T)$ be as defined in \eqref{eq:estimates} and assume $C \in \mathcal{L}(\sX,\sY)$ and $B \in \mathcal{L}(\sU,\dom(A))$. Assume also either (i): $P_0 \in \mathcal{L}(\sX,\dom(A))$, or (ii): $x \in \dom(A)$ almost surely.
Then
\[
\Ex{\norm{\hat z_{T,n}-\hat z(T)}_{\sX}^2} \le \frac{M_1 T^2}{n}+\frac{M_3 T^4}{n^2} + \textup{err}_x
\]
where $M_1$ and $M_3$ are as in Theorem~\ref{thm:noise} end $\textup{err}_x$ is as in Theorem~\ref{thm:conv_bddC} in the case of assumption (i), or as in Theorem~\ref{thm:bddC_nat} in the case of assumption (ii).
\end{corollary}
\noindent The proof is the same as the proof of Theorem~\ref{thm:noise}, with the modifications of Theorems~\ref{thm:conv_bddC} or~\ref{thm:bddC_nat}. Note that $\tr(ABQB^*A^*) \le \norm{A}_{\mathcal{L}(\dom(A),\sX)}^2\norm{B}_{\mathcal{L}(\sU,\dom(A))}^2 \tr(Q)$ and $\tr(BQB^*) \le \norm{B}_{\mathcal{L}(\sU,\sX)}^2 \tr(Q)$.

\section{Discussion}

Since the implementation of the discrete time Kalman filter is
straightforward, it is a tempting choice for state estimation for
discretized continuous time systems.  As the temporal discretization
is refined, the discrete time state estimate converges pointwise to
the continuous time estimate in $L^2(\Omega;\sX)$.  In this article,
we derived convergence speed estimates at which the discrete time Kalman filter estimate converges to the continuous time estimate as the temporal discretization is refined. The result was achieved for both finite and infinite dimensional systems with bounded observation operator and smooth input operator. In the case of infinite dimensional systems, some smoothness assumption on the initial state
is needed for obtaining any convergence speed estimates. This was
demonstrated in Example~\ref{ex:wave}. Possible additional assumptions
are (i): for the initial state covariance it holds that $P_0 \in
\mathcal{L}(\sX,\dom(A))$; or (ii): for the initial state it holds that
$x \in \dom(A)$ almost surely. In the latter case we obtained the same
convergence speed estimate as for finite dimensional systems.

A topic that would require further work are systems with infinite dimensional output
space. The
output space dimension $r$ does not appear explicitly in the
convergence speed estimates, except for
Thm.~\ref{thm:conv_bddC}. However, in the proofs we need an upper
bound for $\norm{\left( C_hP_jC_h^*+\frac{h}{2}R
  \right)^{-1}}_{\mathcal{L}(\sY)}$ and thus, in order to obtain
\eqref{eq:C_R}, we made a coercivity assumption $R \ge \epsilon I>0$
which excludes infinite dimensional output space since $R$ is required
to be a trace class operator. In the beginning, we also assumed that the input space $\sU$ is finite dimensional. This is merely an assumption by which tedious definitions of infinite dimensional Wiener processes are avoided. For more on this subject, we refer to \cite{DaPrato}.

Two more topics that are not covered by this article are the long time
behaviour as $T \to \infty$, and using some approximate time
integration scheme for taking the time step. When $T$ grows, the error
covariance converges under some assumptions on the observability of
the system. Of course,
the observability of the continuous time system does not imply the
observability of the discretized system. In the case where there is
input noise affecting the system, the error covariance limits are
obtained as the solutions $P_d$ and $P_c$ of the corresponding
discrete or continuous time algebraic Riccati equations,
respectively. Then it holds that $\lim_{n \to \infty}\Ex{\norm{\hat
    x_{n\Delta t,n}-\hat x(n \Delta t)}_{\sX}^2}=\tr(P_d-P_c)$ where
$\hat x_{n\Delta t,n}$ and $\hat x(n \Delta t)$ are defined in
\eqref{eq:estimates}.
Finally, further research would be needed
to study the error caused to the state estimate if some numerical time
integration scheme is used for computing the discrete time update,
that is, $e^{A \Delta t}$ is not computed accurately. A similar
problem is addressed in \cite{Axelsson} and \cite{Axelsson2}, but they
are mainly concerned with the stability of the resulting filter.

 \subsection*{Acknowledgements}

 The author was financially supported by the Finnish Graduate School in
 Engineering Mechanics. The author thanks Dr. Jarmo Malinen for
 valuable comments on the manuscript.

 \bibliographystyle{plain}

\end{document}